\documentclass{amsart}

\usepackage{amssymb}
\usepackage{hyperref}
\usepackage{comment}
\usepackage{mathrsfs}
\usepackage{amsmath,amscd}

\newtheorem{theorem}{Theorem}[section]
\newtheorem{lemma}[theorem]{Lemma}
\newtheorem{proposition}[theorem]{Proposition}

\newtheorem*{gluing}{Gluing Lemma}
\newtheorem*{question}{Question}

\theoremstyle{definition}
\newtheorem{definition}[theorem]{Definition}
\newtheorem{example}[theorem]{Example}

\theoremstyle{remark}
\newtheorem{remark}[theorem]{Remark}

\numberwithin{equation}{section}

\def\ds{\rule{0pt}{1.5ex}}

\begin{document}

\title{A gluing construction for polynomial invariants}
\author{Jia Huang}
\address{School of Mathematics, University of Minnesota, Minneapolis, MN 55455, USA}
\email{huang338@umn.edu}
\thanks{The author is grateful to Victor Reiner for asking questions which
led to the research and for giving many constructive suggestions.
The author thanks Aaron Potechin for proving that the ring of
invariants of a seaweed subgroup is polynomial, and thanks the
anonymous referee for providing helpful suggestions.
}

\keywords{Polynomial gluing; Ring of invariants; Sparsity groups}

\date{\today}

\begin{abstract}
We give a {\it polynomial gluing construction} of two groups
$G_X\subseteq GL(\ell,\mathbb F)$ and $G_Y\subseteq GL(m,\mathbb F)$
which results in a group $G\subseteq GL(\ell+m,\mathbb F)$ whose
ring of invariants is isomorphic to the tensor product of the rings
of invariants of $G_X$ and $G_Y$. In particular, this result allows
us to obtain many groups with polynomial rings of invariants,
including all $p$-groups whose rings of invariants are polynomial
over $\mathbb F_p$, and the finite subgroups of $GL(n,\mathbb F)$
defined by sparsity patterns, which generalize many known examples.
\end{abstract}

\maketitle

\section{Introduction}\label{polygluing}

The main theme of this paper is the {\it polynomial gluing
construction} described below. Recall that for a finite dimensional
vector space $V$ over an arbitrary field $\mathbb F$, the symmetric
algebra $S(V)$ is isomorphic to a polynomial algebra, and any
subgroup $G$ of $GL(V)$ acts on $S(V)$.

\begin{definition}\label{gluingdef}
Suppose $V=X\oplus Y$ a direct sum of $\mathbb F$-subspaces $X,Y$,
 $G_X\subseteq GL(X)$, $G_Y\subseteq GL(Y)$ are two groups, and
$\Phi\subseteq{\rm Hom}_{\,\mathbb F}(Y,X)$ is a $(G_X,G_Y)$-subbimodule.
Then one can check
\begin{equation}\label{2form}
G=\left\{g=\begin{bmatrix}
g_{\ds X} &  \phi_g  \\ 0 & g_{\ds Y}
\end{bmatrix}:\
\begin{array}{l}
g_{\ds X}\in G_X,\\
g_{\ds Y}\in G_Y,\\
\phi_g\in \Phi
\end{array}\right\}
\end{equation}
is a subgroup of $GL(V)$ isomorphic to $(G_X\times G_Y)\ltimes\Phi$,
where $g$ acts by
$$g\begin{bmatrix} x \\ y \end{bmatrix}
=\begin{bmatrix}g_{\ds X}(x)+\phi_g(y) \\ g_{\ds Y}(y)\end{bmatrix},
\textrm{ for all }\begin{bmatrix} x \\ y \end{bmatrix}\in X\oplus
Y.$$ Call $G$ a {\it polynomial gluing} of $G_X$ and $G_Y$ through
$\Phi$ (or really through $\pi$) if there exists a $G_X\times
G_Y$-equivariant isomorphism $\pi:S(V)\rightarrow S(V)^{\Phi}$ of
$\mathbb F$-algebras, where $\Phi$ acts on $V$ via the embedding
$$\phi_g\mapsto
\begin{bmatrix} 1 &  \phi_g  \\ 0 & 1
\end{bmatrix}.$$
In particular, this implies $S(V)^\Phi$ is itself a polynomial
algebra.
\end{definition}

\begin{proposition}\label{gluing invariants}
If $G$ is a polynomial gluing of $G_X$ and $G_Y$ through $\pi$, then
$$S(V)^G=\pi\left(S(X)^{G_X}\otimes S(Y)^{G_Y}\right)\quad{\rm
inside\ } S(V)^\Phi.$$ In particular, since $\pi$ is an isomorphism,
one has a ring isomorphism
$$S(V)^G\cong S(X)^{G_X}\otimes S(Y)^{G_Y}$$
ignoring gradings.
\end{proposition}

\begin{proof}
Note that $\Phi$ is a normal subgroup of $G$ with quotient $G/\Phi\cong
G_X\times G_Y$. Then
\begin{eqnarray*}
S(V)^G&=&(S(V)^{\Phi})^{G_X\times G_Y}
=(\pi(S(V)))^{G_X\times G_Y}\\
&=&\pi\left(S(X \oplus Y)^{G_X\times G_Y}\right)\\
&=&\pi\left(S(X)^{G_X}\otimes S(Y)^{G_Y}\right)
\end{eqnarray*}
since $\pi$ is a $G_X\times G_Y$-equivariant isomorphism.
\end{proof}

\begin{example}\label{ex1}
Let $V$ be a vector space over $\mathbb F=\overline{\mathbb F}_q$
with basis $\{x,y\}$, let $X=\mathbb Fx$, $Y=\mathbb Fy$, so that
$V=X\oplus Y$, and define
\begin{eqnarray*}
G_X&=&\mathbb F_{q^a}^\times\subseteq GL(X),\\
G_Y&=&\mathbb F_{q^b}^\times\subseteq GL(Y),\\
\Phi&=&\left\{\phi:\phi(y)=\gamma x,\ \gamma\in\mathbb F_{q^{ab}}\right\}\\
&\cong&\left\{\begin{bmatrix} 1 & \gamma \\ 0 & 1 \end{bmatrix}:
\gamma\in\mathbb F_{q^{ab}}\right\}
\cong\mathbb F_{q^{ab}},\\
G&=&\left\{\begin{bmatrix}
\alpha & \gamma \\
0 & \beta
\end{bmatrix}:
\alpha\in\mathbb F_{q^a}^\times, \beta\in\mathbb F_{q^b}^\times,
\gamma\in\mathbb F_{q^{ab}}\right\}.
\end{eqnarray*}
 Then one can check that $G$
is a polynomial gluing of $G_X$ and $G_Y$ through $\Phi$, with
\begin{eqnarray*}
\pi\,:\, S(V)&\rightarrow& S(V)^\Phi,\\
x&\mapsto& x,\\
y&\mapsto& f,
\end{eqnarray*}
where $f=y^{q^{ab}}-x^{q^{ab}-1}y$, and that
\begin{eqnarray*}
S(V)^G&=&\mathbb F[x^{q^a-1},f^{q^b-1}]\\
&=&\pi\left(\mathbb F[x^{q^a-1},y^{q^b-1}]\right)\\
&=&\pi\left(S(X)^{G_X}\otimes S(Y)^{G_Y}\right).
\end{eqnarray*}
\end{example}

Proposition \ref{gluing invariants} shows that polynomial gluings
are relevant to the {\it Polynomial Algebra Problem} - determine all
finite groups with polynomial invariant rings - for which the
nonmodular case was answered by the Chevalley--Shephard--Todd
theorem and the irreducible case was solved by Kemper and Malle
\cite{irreducible}. In the remaining case, that is, $G$ is reducible
and modular, the polynomial gluing construction can provide
interesting examples. First note that the construction can be
iterated in the following way.

\begin{definition}
A group $G\subseteq GL(V)$ is an {\it iterated polynomial gluing}
of $G_{X_1},\ldots,G_{X_t}$ if recursively $G$ is a polynomial gluing
of $G_X$ and $G_Y$ as above, where $G_X$ is an iterated polynomial
gluing of $G_{X_1},\ldots,G_{X_s}$ and $G_Y$ is an iterated polynomial
gluing of $G_{X_{s+1}},\ldots,G_{X_t}$ for some integer $s$.
\end{definition}

In Section \ref{PAP} we identify a family of $p$-groups given by
Nakajima \cite{Nakajima2} as iterated polynomial gluings. In
particular, it implies the following.

\begin{proposition}\label{gluingof1}
A $p$-group in $GL(n,\mathbb F_p)$ has a polynomial ring of
invariants over $\mathbb F_p$ if and only if it is an iterated
polynomial gluing of $n$ copies of trivial groups $\{1_{\mathbb
F_p}\}$.
\end{proposition}

\noindent However, there exist groups with polynomial ring of invariants that
can not be constructed by polynomial gluing; see Example \ref{Sn}.

Another example (which originally motivated the polynomial gluing
construction for us) is the following.

\begin{definition}\label{supdef}
Let $V$ be the defining representation of $GL(n,\mathbb F)$ with
basis $\{x_1,\ldots,x_n\}$. Call a map
$$\sigma:\{1,\ldots,n\}\times\{1,\ldots,n\}\rightarrow 2^{\mathbb
F}\setminus\{\emptyset\}$$ a {\it sparsity pattern}. Given a
sparsity pattern $\sigma$, define the {\it sparsity group}
$GL^\sigma(n,\mathbb F)$ to be the subgroup of $GL(n,\mathbb F)$
generated by all transvections
$$\begin{aligned}
T_{ij}(a):\ & x_j\mapsto x_j+ax_i,\\
& x_k\mapsto x_k,\ {\rm for}\ k\ne j\\
{\rm with}\ \ \ & a\in\sigma(i,j),\ 1\leqslant i\ne j\leqslant n,
\end{aligned}$$
as well as all diagonal matrices
$$\begin{aligned}
D_i(a):\  & x_i\mapsto ax_i,\\
& x_k\mapsto x_k,\ {\rm for}\ k\ne i\\
\textrm{with }\ \  & a\in\sigma(i,i)\setminus\{0\},\ 1\leqslant i\leqslant n.
\end{aligned}$$
\end{definition}

For instance, the group discussed in Example \ref{ex1} is
$GL^\sigma(2,\mathbb F_{q^{ab}})$ where $\sigma$ is given by
$$\begin{matrix}
\sigma(1,1)=\mathbb F_{q^a}. & \sigma(1,2)=\mathbb F_{q^{ab}},\\
\sigma(2,1)=0,                             & \sigma(2,2)=\mathbb F_{q^b}.
\end{matrix}$$
In Section \ref{supsigma} we investigate the structure of an
arbitrary finite sparsity group $GL^\sigma(n,\mathbb F)$ and
identify it as the following iterated polynomial gluing.

\begin{theorem}\label{supgluing}
If the sparsity group $G=GL^\sigma(n,\mathbb F)$ is finite, then it is
an iterated polynomial gluing of subgroups of $GL(m,\mathbb F)$ of
the following two types:
\begin{enumerate}
\item[(i)] those with polynomial rings of invariants for $m=1,2$,
and

\item[(ii)] those between $SL(m,\mathbb F_q)$ and $GL(m,\mathbb
F_q)$ for various integers $m\in\{3,4,\ldots,n\}$ and finite fields
$\mathbb F_q\subseteq \mathbb F$.
\end{enumerate}

\noindent Consequently the ring of invariants of $G$ is polynomial.
\end{theorem}

This theorem covers many known results asserting that various groups
have polynomial invariants, including the following.

\begin{itemize}
\item The general linear group $GL(n,\mathbb F_q)$ and the special
linear group $SL(n,\mathbb F_q)$ (Dickson \cite{Dickson}),

\item The group of all upper triangular matrices with $1$'s on diagonal
(M.-J. Bertin \cite{Bertin}),

\item Parabolic subgroups $G_{n_1,\ldots,n_{\ell}}$
of all block upper triangular matrices with diagonal blocks of size
$n_1,\ldots,n_{\ell}$ (Hewett \cite{parabolic}, Mui \cite{Mui}),

\item Seaweed groups $G_{\alpha,\beta}=G_{\alpha}\cap
(G_{\beta})^t$ associated with seaweed Lie algebras
\cite{seaweed1,seaweed2}, where $\alpha=(\alpha_1,\ldots,\alpha_r)$,
$\beta=(\beta_1,\ldots,\beta_s)$ are integer compositions of $n$
(Potechin\footnote{REU 2008 (mentored by V. Reiner and D. Stanton),
School of Mathematics, University of Minnesota.}),

\item The groups $G(r)$ of all matrices that agree
with $I_n$ in its first $r$ rows (Steinberg \cite{Steinberg2}),

\item The groups $E_{\mathbb F_q}(r)$ generated by
$\{T_{in}(\alpha) :\alpha\in \mathbb F_q,i=1,\ldots,r\}$ as well as
the transpose groups $E_{\mathbb F_q}(r)^t$ for $r=1,\ldots,n-1$,
(Smith \cite[Proposition 8.2.5, 8.2.6]{Smith}).
\end{itemize}

\begin{remark}
There is another natural way to get sparsity groups in
$GL(n,\mathbb F)$. Given a sparsity pattern $\sigma$, let
$$GL_{\sigma}(n,\mathbb F)=\left\{[a_{ij}]_{i,j=1}^n\in GL(n,\mathbb F):
a_{ij}\in\sigma(i,j),i,j=1,\ldots,n\right\}.$$ One can show that
$GL_\sigma(n,\mathbb F)$ is a group if and only if it equals
$GL^{\sigma}(n,\mathbb F)$ and hence it suffices to consider only
$GL^\sigma(n,\mathbb F)$; for this reason we omit the discussion of
$GL_\sigma(n,\mathbb F)$.
\end{remark}

\section{Gluing Lemma}\label{Gluing lemma}

Using a common technique for showing a ring of invariants to be
polynomial we obtain a special kind of polynomial gluing
construction. This turns out to be sufficient for our later use
in proving Proposition \ref{gluingof1} and Theorem \ref{supgluing}.

\begin{lemma}{\cite[Proposition 16]{AlgIndep}}\label{criterion} 
Let $V$ be a vector space over an arbitrary field $\mathbb F$ with
basis $\{x_1,\ldots,x_n\}$, and $G\subseteq GL(V)$ a finite group.
Then $\mathbb F[x_1,\ldots,x_n]^G$ is a polynomial algebra if and
only if there are algebraically independent homogeneous invariants
$f_1,\ldots,f_n\in \mathbb F[x_1,\ldots,x_n]^G$ such that
$\deg(f_1)\cdots\deg(f_n)=|G|$.
\end{lemma}

\begin{gluing}
Let $\mathbb F_q\subseteq\mathbb F$ be a field extension, and
consider a finite dimensional $\mathbb F$-space decomposition
$V=X\oplus Y$ with $Y$ being $\mathbb F_q$-rational, i.e.
$Y=Y'\otimes_{\mathbb F_q}\mathbb F$ for some $\mathbb F_q$-space
$Y$. Let $X'\subseteq X$ be an $\mathbb F_q$-subspace stabilized by
some subgroup $G_X\subseteq GL_{\mathbb F}(X)$.

Then any subgroup $G_{Y'}\subseteq GL_{\mathbb F_q}(Y')$ gives rise
to a polynomial gluing of $G_X$ and
$$G_Y=G_{Y'}\otimes1_{\mathbb F}$$ inside $GL_{\mathbb F}(Y)$ through
$$\Phi={\rm Hom}_{\mathbb F_q}(Y',X')\otimes 1_{\mathbb F}$$ inside
${\rm Hom}_{\mathbb F}(Y,X)$.
\end{gluing}

\begin{proof}
One easily sees $G_X\circ\Phi\circ G_Y\subseteq\Phi$ from our
assumptions and thus (\ref{2form}) defines a group $G$. Consider the
polynomial $$P(t)=\prod_{x\in X'}(t-x)$$ which is well-known to be
an $\mathbb F_q$-linear function of $t$ in $S(X)^{G_X}[t]$ (see, for
example, Wilkerson \cite{Wilkerson}). Let $\{x_1,\ldots,x_\ell\}$ be
an $\mathbb F$-basis for $X$ and $\{y_1,\ldots,y_m\}$ an $\mathbb
F_q$-basis for $Y'$ (hence an $\mathbb F$-basis for $Y$). Define
$f_i=P(y_i)$ in $S(V)$ for $i=1,2,\ldots,m$. One can check
$x_1,\ldots,x_\ell,f_1,\ldots,f_m$ are algebraically independent
$\Phi$-invariants.
Also note that
$$\deg(x_1)\cdots\deg(x_\ell)\cdot\deg(f_1)\cdots\deg(f_m)=|X'|^m=|\Phi|.$$
Thus Lemma \ref{criterion} gives
$$\mathbb F_q[V]^{\Phi} =\mathbb F_q[x_1,\ldots,x_\ell,f_1,\ldots,f_m].$$
It follows that the map
$$\begin{array}{ccccc}
\pi:& S(V)&\rightarrow& S(V)^\Phi,\\
&x_i&\mapsto& x_i,& i=1,\ldots,\ell,\\
&y_i&\mapsto& f_i, & i=1,\ldots,m
\end{array}$$
is an isomorphism of $\mathbb F_q$-algebras, and one checks that it
is $G_X\times G_Y$-equivariant as follows: for every $(g_{\ds
X},g_{\ds Y})\in G_X\times G_Y$, one has
$$\pi(g_{\ds X}x)=g_{\ds X}x=g_{\ds X}\pi(x),\ \textrm{for all } x\in X,$$
$$\pi(g_{\ds Y}y)=P(g_{\ds Y}y)=g_{\ds Y}P(y)=g_{\ds Y}\pi(y),\textrm{ for all } y\in Y,$$
using the $\mathbb F_q$-linearity of $P(t)$.
\end{proof}

The above lemma is only a special case of the polynomial gluing
construction from Definition \ref{gluingdef}, as demonstrated by the following example.

\begin{example}[R.E. Stong]
Given a prime number $p$, let $\{1,\alpha,\beta\}$ be a basis for
$\mathbb F=\mathbb F_{p^3}$ over $\mathbb F_p$. The group $G$ generated by
$$\begin{bmatrix}
1&1 \\
 &1 \\
 & &1
 \end{bmatrix},
\begin{bmatrix}
1& &1\\
 &1  \\
 & &1
\end{bmatrix},
\begin{bmatrix}
1&\alpha&\beta\\
 & 1          \\
 &      & 1
\end{bmatrix}$$
naturally acts on $V=\mathbb F^3$ with basis $\{x_1,x_2,x_3\}$ and
has a polynomial ring of invariants \cite[\S6.3, Example 2]{NS}. Let
$X=\mathbb Fx_1\oplus\mathbb Fx_2$ and $Y=\mathbb Fx_3$. It is easy
to see that $G$ is a polynomial gluing of $G_X=\{1\}$ and $G_Y=\{I_2\}$,
since the isomorphism $\pi$ sending $\{x_1,x_2,x_3\}$ to a set of algebra
generators of $S(V)^G$ is trivially equivariant. But the Gluing Lemma
does not apply to it, as one can check $\Phi\ne{\rm Hom}_{\mathbb F_q}(Y',X')$
for all possible choices of $X'$, $Y'$, and $\mathbb F_q$.
\end{example}

\section{The Polynomial Algebra Problem}\label{PAP}

We first consider Nakajima's $p$-groups and prove Proposition
\ref{gluingof1}.

\begin{theorem}[Nakajima \cite{Nakajima2}]\label{pgroup}
Let $\rho:G\hookrightarrow GL(V)=GL(n,\mathbb F)$ be a
representation of a finite $p$-group $G$ over a field $\mathbb F$ of
characteristic $p>0$.

Suppose that there is a basis $x_1,\ldots,x_n$ for $V$ such that
$\prod_{i=1}^n|Gx_i|=|G|$ and $\oplus_{i=1}^j \mathbb Fx_i$ is an
$\mathbb FG$-submodule of $V$ for $j=1,\ldots,n$. Then $S(V)^G$ is a
polynomial algebra generated by the top Chern classes $\prod_{x\in
Gx_i}x$ of the orbits $Gx_i$ for $i=1,\ldots,n$.

Furthermore, the converse holds when $\mathbb F=\mathbb F_p$, that
is, if $G$ is a $p$-group inside $GL(n,\mathbb F_p)$ whose ring of
invariants is polynomial, then there exists a basis $x_1,\ldots,x_n$
as above.
\end{theorem}

\begin{proposition}\label{gluing1}
Every $p$-group $G$ as in Theorem \ref{pgroup} is an iterated
polynomial gluing of $n$ copies of the trivial group $\{1_\mathbb F\}$.
\end{proposition}

\begin{proof}
Since $\oplus_{i=1}^j \mathbb Fx_i$ is an $\mathbb FG$-submodule of
$V$ for $j=1,\ldots,n$, all elements in $G$ are in upper triangular
matrix form under $x_1,\ldots,x_n$. The group $G_j$ of the $j$-th
diagonal entries of all elements in $G$ is canonically embedded in
$G$ and hence has a power of $p$ as its order. On the other hand,
since $G_j$ is a finite group, the field generated by $G_j$ must be
finite, and so the order of $G_j$ divides $p^k-1$ for some integer
$k$. This forces $G_j$ to be trivial. Therefore the diagonal entries
of elements in $G$ are all $1$'s.

Let $$H=\{h\in GL(n,\mathbb F):hx_j\in Gx_j,j=1,\ldots,n\}.$$ It is
clear that $G\subseteq H$ and by considering the number of choices
of the column vectors of $h\in H$ we get $|H|\leqslant
|Gx_1|\cdots|Gx_n|=|G|$. Therefore $G=H$ and it follows that the restriction
of $G$ to $X=\mathbb Fx_1\oplus\cdots\oplus\mathbb Fx_{n-1}$ is
$$G_X=\left\{g\in GL(n-1,\mathbb F): gx_j\in Gx_j,j=1,\ldots,n-1\right\}.$$
From this one checks that $G_X$ satisfies the properties for $G$ in Theorem
\ref{pgroup} with $n$ replaced by $n-1$, and thus is an iterated polynomial
gluing of $n-1$ copies of $\{1_{\mathbb F}\}$ by induction.

Let $Y=\mathbb Fx_n$ and for each $g$ in $G$, define $\phi_g\in{\rm
Hom}_{\mathbb F}(Y,X)$ by $\phi_g(x_n)=gx_n-x_n$. We wish to show
that $\Phi=\left\{\phi_g: g\in G\right\}$ is an $\mathbb
F_p$-subspace of ${\rm Hom}_{\mathbb F}(Y,X)$, i.e.
$\phi_g+\phi_h\in\Phi$ for all $g,h\in G$. Since $G=H$, there exist
$g',h'\in G$ with $g'x_j=h'x_j=x_j$ for $j=1,\ldots,n-1$, and
$g'x_n=gx_n$, $h'x_n=hx_n$. Thus
\begin{eqnarray*}
g'h'x_n&=&g'hx_n\\
&=&g'(x_n+\phi_h(x_n))\\
&=&gx_n+g'\phi_h(x_n)\\
&=&x_n+\phi_g(x_n)+\phi_h(x_n) \end{eqnarray*} and so
$\phi_g+\phi_h=\phi_{g'h'}\in\Phi$. Therefore $\Phi$ is an $\mathbb
F_p$-subspace of ${\rm Hom}_{\mathbb F}(Y,X)$ and must be equal to
${\rm Hom}_{\mathbb F_p}(Y',X')\otimes_{\mathbb F_p}1_{\mathbb F}$,
where $Y'=\mathbb F_p x_n$ and $X'=\{\phi_g(x_n):g\in G\}$, since $Y'$ is
of dimension one. It follows
from the Gluing Lemma that $G$ is a polynomial gluing of $G_X$ and
$G_Y=\{1_{\mathbb F}\}$ and we are done by induction on $n$.
\end{proof}

Combining Theorem \ref{pgroup} and Proposition \ref{gluing1} one
immediately has Proposition \ref{gluingof1}, which gives a simple
reason for a $p$-group to have a polynomial ring of invariants over
$\mathbb F_p$: the ring of invariants of the trivial group
$\{1_{\mathbb F_p}\}$ is polynomial. Of course, it is easy to prove
that these $p$-groups have polynomial rings of invariants directly
by using Lemma \ref{criterion}; we do not gain a truly simpler proof
by the method of polynomial gluing.

The next example shows that iterating the polynomial gluing
construction cannot produce all finite groups with polynomial rings
of invariants from the irreducible ones.

\begin{example}\label{Sn}
The symmetric group $\mathfrak S_n$ acts on $V$ by permuting the basis
$\{x_1,\ldots,x_n\}$, and the ring of invariants $S(V)^{\mathfrak S_n}$ is
always polynomial. One can show that the only nonzero proper
subrepresentations of $V$ are the line $U$ spanned by $x_1+\ldots+x_n$
and the hyperplane $W$ spanned by $x_1-x_2,x_2-x_3,\ldots,x_{n-1}-x_n$.
It is clear that $V=U\oplus W$ if and only if ${\rm char}(\mathbb F)=p\nmid n$.
We claim that if $p\mid n$ then $\mathfrak S_n$ can not be obtained by
a polynomial gluing construction unless $p=n=2$.

In fact, when $p=n=2$ taking $X=U=W=\mathbb F_2(x_1+x_2)$ and
$Y=\mathbb F_2(x_1+ax_2)$ for some $a\ne 1$ one has
$$\mathfrak S_2=\left\{1,(12)=\begin{bmatrix}1&1+a\\0&1\end{bmatrix}\right\}$$
is a polynomial gluing of two copies of $\{1\}$.

Assume for the sake of contradiction $p\mid n\geqslant 3$ and
$G=\mathfrak S_n$ is a polynomial gluing of $G_X$ and $G_Y$.
Then $X$ is either $U$ or $W$.

If $X=U$ then $Y$ is isomorphic to the quotient representation $V/U$
of $\mathfrak S_n$, which is easily seen to be faithful. Hence
$$|G_Y|=|G|=|G_X||G_Y||\Phi|$$ and thus $\Phi=\{0\}$, which implies
$Y$ is a subrepresentation of $V$. Then $Y$ is either $U$ or $W$,
which is absurd. However, $S(Y)^{G_Y}$ is polynomial: one has
$S(V/U)^{\mathfrak S_n}=\mathbb F[\bar e_2,\ldots,\bar e_n]$ where
$e_i$ is the $i$-th elementary symmetric polynomial in
$x_1,\ldots,x_n$ and $\bar e_i=e_i+U$.

If $X=W$ then the same argument as above
leads to a contradiction, since the subrepresentation $W$ is also faithful.
Neither $S(X)^{G_X}=S(W)^{\mathfrak S_n}$ nor $S(W/U)^{\mathfrak S_n}$
is polynomial for $n\geqslant5$ \cite[Section 5]{irreducible}.
\end{example}

This suggests the following question, for which we currently have no
answer.

\begin{question}
Let $G\subseteq GL(n,\mathbb F)$ be a group whose order is divisible
by ${\rm char}(\mathbb F)$ and whose ring of invariants is
polynomial. If $G$ is primitive and has a stable subspace $X$ of
dimension $m$ in $V$, then is there a complement space $Y$ such that
$V=X\oplus Y$ and $G$ is a polynomial gluing of $G_X\subseteq
GL(m,\mathbb F)$ and $G_Y\subseteq GL(n-m,\mathbb F)$?
\end{question}

\section{Sparsity groups $GL^\sigma(n,\mathbb F)$}\label{supsigma}

Let $V$ be the defining representation of $GL(n,\mathbb F)$ with basis
$\{x_1,\ldots,x_n\}$. Recall from Definition \ref{supdef}
that a sparsity pattern $\sigma$ gives rise to a group
$GL^\sigma(n,\mathbb F)$ sitting inside $GL(n,\mathbb F)$. First
note that one can order $x_1,\ldots,x_n$ to make all elements in
$GL^\sigma(n,\mathbb F)$ have a upper triangular block form.

\begin{definition}\label{preorder}
Define a preorder on $\{1,\ldots,n\}$ by saying $i\leqslant_\sigma
j$ if there exists a sequence $i=i_0,i_1,\ldots,i_n=j$ such that
either $i_k=i_{k+1}$ or $\sigma(i_k,i_{k+1})\ne\{0\}$ for
$k=0,1,\ldots,n-1$. This preorder induces an equivalence relation:
$i\sim_\sigma j$ if $i\leqslant_\sigma j$ and $j\leqslant_\sigma i$;
denote the equivalence classes by $J_1,\ldots,J_t$, and their sizes
by $n_1,\ldots,n_t$. By permuting $x_1,\ldots,x_n$ we assume without
loss of generality that the preorder ``$\leqslant_\sigma$'' is {\it
natural}, i.e. $i\leqslant_\sigma j$ only if $i<j$ or $i\sim_\sigma
j$. Define $X_k=\bigoplus_{j\in J_k}\mathbb Fx_j$.
\end{definition}


\begin{proposition}\label{irreducible}
Let $G=GL^\sigma(n,\mathbb F)$. Then $X_1\oplus\cdots\oplus X_k$ is
$G$-stable for $k=1,\ldots,t$; consequently all elements in $G$ have
the upper triangular block form
\begin{equation}\label{form}
M=\left[\begin{array}{ccccc}
M_{11} & M_{12} & M_{13} & \cdots & M_{1t}  \\
   0   & M_{22} & M_{23} & \cdots & M_{2t}  \\
       &        & \cdots &        &         \\
       &        & \cdots &        &         \\
   0   &    0   & \cdots &    0   & M_{tt}
\end{array}\right]\end{equation}
where each $M_{rs}$ is an $n_r\times n_s$ matrix. Moreover, $G$ is
irreducible if and only if $t=1$. \hfill$\Box$
\end{proposition}

\begin{proof}
The assertions are all straightforward except the irreducibility of
$G$ when $t=1$. To see it, assume $n\geqslant 2$ and let $W$ be a
nonzero $\mathbb F$-subspace of $V$ that is stable under $G$. Then
$W$ contains a nonzero element $v=\sum_{i=1}^n c_ix_i$ with
$c_i\ne0$ for some $i$. Since $t=1$ and $n\geqslant 2$, there exists
a $j\ne i$ such that $\sigma(j,i)$ contains a nonzero element $a$.
One checks that
$$T_{ji}(a)v-v=(ac_i)x_j\in W$$
which implies $x_j\in W$ since $ac_i\in\mathbb F^\times$.
Repeating this one has $x_1,\ldots,x_n\in W$, since $t=1$, and hence
$V=W$.
\end{proof}

Now we investigate the structure of a {\it finite} sparsity group
$G=GL^\sigma(n,\mathbb F)$ and use it to prove Theorem
\ref{supgluing}. If $\sigma(i,j)=\{0\}$ whenever $1\leqslant i\ne
j\leqslant n$, then $G$ is simply a direct sum of finite subgroups
of $\mathbb F^\times$, and Theorem \ref{supgluing} clearly holds in
this case.

\begin{theorem}\label{structure}
Assume $\sigma(i,j)\ne\{0\}$ for some $(i,j)$ with $1\leqslant i\ne
j\leqslant n$, and $G=GL^\sigma(n,\mathbb F)$ is finite. Then by
rescaling basis for $V$ one can find additive groups
$k_{rs}$ in $\mathbb F$, $1\leqslant r,s\leqslant t$, such that
\begin{enumerate}
\item[(a)] $k_{rr}$ is finite field for $r=1,\ldots,t$,

\item[(b)] $k_{rs}=\{0\}$ if $r>s$,

\item[(c)] $k_{rs}k_{s\ell}\subseteq k_{r\ell}$ if $r\leqslant s\leqslant\ell$,

\item[(d)] $G=\left\{[M_{rs}]_{r,s=1}^t:\begin{array}{cc}
M_{rr}\in G_{X_r}, &1\leqslant r\leqslant t,\\
M_{rs}\in M(n_r\times n_s, k_{rs}),&1\leqslant r<s\leqslant t
\end{array}\right\}$,
\end{enumerate}
where
\begin{enumerate}
\item[(e)] $SL(n_r,k_{rr})\leqslant G_{X_r}\leqslant GL(n_r,k_{rr})$
unless $n_r=2$,

\item[(f)] $M(n_r\times n_s, k_{rs})$ is the space of all
$n_r\times n_s$ matrices with entries in $ k_{rs}$.
\end{enumerate}
\end{theorem}

\begin{proof}
By the assumption $G$ contains $T_{ij}(a)$ for some $a\ne 0$ and
thus contains $T_{ij}(ma)=T_{ij}(a)^m$ for all integers $m$. The
finiteness of $G$ forces $ma=0$ for some $m$, i.e. ${\rm
char}(\mathbb F)=p>0$.

Let
$$S(i,j)=\left\{\begin{array}{ll}
\{a: T_{ij}(a)\in G\}, & {\rm if}\ 1\leqslant i\ne j\leqslant n,\\
\{a:D_i(a)\in G\}, & {\rm if}\ 1\leqslant i=j\leqslant n.
\end{array}\right.$$
One has $S(i,j)=0$ by (\ref{form}) if $i\in J_r$, $j\in J_s$, $r>s$,
and
\begin{equation}\label{transitivity}
S(i,j)S(j,k)\subseteq S(i,k),\ \ {\rm if}\ i\ne k
\end{equation} by the following equalities:
$$T_{ik}(ab)=\left\{\begin{array}{ll}
\left[T_{ij}(a),T_{jk}(b)\right], & {\rm if}\ i\ne j\ne k,\\
D_i(a)T_{ik}(b)D_i(a^{-1}), &{\rm if}\ i=j\ne k,\\
D_k(b^{-1})T_{ik}(a)D_k(b), &{\rm if}\ i\ne j=k,
\end{array}\right.$$

Define a directed graph on the vertices $1,2,\ldots,n$ with edges
$i\rightarrow j$ if $i\ne j$ and $S(i,j)\ne\{0\}$. By
(\ref{transitivity}), one has the transitivity that $i\rightarrow
j\rightarrow k$ implies $i\rightarrow k$ if $i\ne k$. Hence each
$J_s$ induces a complete subgraph (i.e. a graph with edges in both
direction between any pair of distinct vertices); in particular there is a
directed cycle passing through all vertices in $J_s$ (once for each)
whenever $n_s\geqslant2$. Fixing an $i\in J_r$ with $r<s$ and
letting $(j,k)$ vary along the edges of this cycle one obtains from
(\ref{transitivity}) a corresponding cycle of set-inclusions of the
form
$$S(i,j)\cdot a_{jk}\subseteq S(i,k)$$
with nonzero transition scalars $a_{jk}\in S(j,k)$. The finiteness
of the group $G$ forces all the above inclusions to be equalities,
and one can easily rescale the basis for $X_s$ to make all but one
of the transition scalars $a_{jk}$ to be $1$ so that $S(i,j)=S(i,k)$
for all $j,k\in J_s$. A similar argument clearly works for $S(i,j)$
with $i$ varying in $J_r$ and $j$ fixed in $J_s$, $r<s$. In case
$r=s$ it still works as long as $i\ne j$ and $n_r>2$ so that one is
able to choose three distinct indices in (\ref{transitivity}).
Combining these together one sees that for any fixed pair $(r,s)$,
$S(i,j)$ takes the same set for all pairs $(i,j)$ with $i\in J_r$,
$j\in J_s$, $i\ne j$, unless $r=s$ and $n_r=2$.

Now we can find the additive groups $k_{rs}$ and verify
the assertions (a)--(f).

It is clear that we should take $k_{rs}=\{0\}$ if $r>s$, so that
assertion (b) is true.

If $r<s$ then let $k_{rs}=S(i,j)$ for any $i\in J_r$, $j\in J_s$,
which is clearly closed under addition.

If $r=s$ then we consider the following cases.

\vskip3pt\noindent{\sf Case 1}. If $n_r\geqslant 3$ then let
$k_{rr}=S(i,j)$ for any $i,j\in J_r$ with $i\ne j$. It is closed
under addition since $T_{ij}(a)T_{ij}(b)=T_{ij}(ab)$, and closed
under multiplication by (\ref{transitivity}) applied to distinct
$i,j,k\in J_r$. Thus $k_{rr}$ is a ring (not necessarily unital),
and the finiteness of $G$ forces $k_{rr}$ to be a finite field.
Hence (a) holds for this case.

Then one has $1\in S(i,j)$ and $S(j,j)\subseteq S(i,j)=k_{rr}$ by
(\ref{transitivity}). It follows from Lemma \ref{generating GL}
below that $G_{X_r}$ consists of all $n_r\times n_r$ matrices with
determinant in the multiplicative group generated by $\cup_{j\in
J_r}\sigma(j,j)\setminus\{0\}$. Hence (e) holds for this case.

\vskip3pt\noindent{\sf Case 2}. If $n_r=2$, i.e. $J_r=\{i,j\}$, then
let $k_{rr}$ be the field generated by $S(i,j)$ and $S(j,i)$. To
show it is finite, it suffices to show any nonzero element $a\in
S(i,j)\cup S(j,i)$ satisfies a polynomial equation. By the rescaling
of basis for $X_r$ that has been used before, one may assume $1\in
S(i,j)$ and $0\ne a\in S(j,i)$, without loss of generality. By
induction one sees that the $(j,i)$-entry of
$(T_{ij}(1)T_{ji}(a))^d$  is a monic polynomial $f_d(a)$ of degree
$d$, and the finiteness of $G$ forces $f_d(a)=a$ for a sufficiently
large $d$. Therefore (a) is true for this case.

\vskip3pt\noindent{\sf Case 3}. If $J_r=\{i\}$ then let $k_{rr}$ be
the field generated by $\sigma(i,i)$, which is finite since ${\rm
char}(\mathbb F)>0$ and $G$ is finite. It is obvious that the assertions
(a) and (e) are true in this case. \vskip3pt

It follows easily from (\ref{transitivity}) and the above definition of
$k_{rr}$ that (c) holds: $k_{rs}k_{s\ell}\subseteq k_{r\ell}$.

It remains to show (d) and (f). The inclusion ``$\subseteq$'' in (d)
is clear (consider the generators of $G$). Conversely, note that the
right hand side of (d) is generated by all $D_r(M_{rr})$ and
$T_{rs}(M_{rs})$ with $M_{rr}\in G_{X_r}$, $M_{rs}\in M(n_r\times
n_s, k_{rs})$, $1\leqslant r<s\leqslant t$, where $D_r(M_{rr})$ and
$T_{rs}(M_{rs})$ both have the block form (\ref{form}) such that the
$r$-th diagonal block of $D_r(M_{rr})$ is $M_{rr}$, the
$(r,s)$-block of $T_{rs}(M_{rs})$ is $M_{rs}$,  and all other blocks
are either zero (if off diagonal) or identity (if on diagonal). It
suffices to show all these $D_r(M_{rr})$ and $T_{rs}(M_{rs})$ belong
to $G$.

If $M_{rr}\in G_{X_r}$ then $G$ contains an element $M$ whose $r$-th
diagonal block is $M_{rr}$. Writing $M$ as a product of the
generators for $G$,  one sees that the only way to alter the $r$-th
diagonal block is to multiply by $T_{ij}(a)$ or $D_i(a)$ with $i,j$
both in $J_r$. Hence $D_r(M_{rr})$ is a product of these generators
and must belong to $G$.

If $M_{rs}\in M(n_r\times n_s,k_{rs})$ then
$T_{rs}(M_{rs})=[a_{ij}]_{i,j=1}^n$ is a product of $T_{ij}(a_{ij})$
as $i,j$ run through $J_r,J_s$, respectively. The definition of $
k_{rs}=S(i,j)$ implies that all $T_{ij}(a_{ij})$ belong to $G$ and
so does the product $T_{rs}(M_{rs})$.

We have verified all the assertions and the proof of the theorem is now complete.
\end{proof}

\begin{lemma}\label{generating GL}
If $K$ is a subgroup of $\mathbb F^\times$ generated by $S$, then
the group $G$ of all $n\times n$ matrices with entries in $\mathbb
F$ and determinant in $K$ is generated by all transvections
$T_{ij}(a)$ with $a\in\mathbb F$ and $1\leqslant i\ne j\leqslant n$
as well as all diagonal matrices $D_i(a)$ with $a\in S$ and
$1\leqslant i\leqslant n$.
\end{lemma}

\begin{proof}
It is well known that $SL(n,\mathbb F)$ is generated by the
transvections $T_{ij}(a)$ with $a\in\mathbb F$ and $1\leqslant i\ne
j\leqslant n$ (see, for example, \cite{Litoff}), and hence contained
in the group $H$ generated by these transvections together with the
diagonal matrices $D_i(a)$ with $a\in S$ and $1\leqslant i\leqslant
n$. One also has $D_1(b)\in H$ for any $b\in K$. Therefore $H=G$.
\end{proof}

\begin{proof}[Proof of Theorem \ref{supgluing}]
It is clear from Theorem \ref{structure} that the Gluing Lemma in
Section \ref{Gluing lemma} applies to $G$ with $\mathbb F_q=k_{tt}$
and
\begin{eqnarray*}
X&=&X_1\oplus\cdots\oplus X_{t-1},\\
Y&=&X_t,\\
Y'&=&\bigoplus_{j\in J_t}k_{tt}x_j,\\
X'&=&\bigoplus_{r=1}^{t-1}\bigoplus_{i\in J_r}k_{rt}x_i.
\end{eqnarray*}
By induction on $t$ one shows that $G$ is an iterated polynomial
gluing of $G_{X_1},\ldots,G_{X_t}$. Then Theorem \ref{supgluing}
follows from Proposition \ref{gluing invariants}, Theorem
\ref{structure} (e), and the following result of Nakajima (see also
Kemper and Malle \cite[Proposition 7.1]{irreducible}).
\end{proof}

\begin{theorem}\cite[Theorem 5.1]{Nakajima1}\label{degree2}
If $G\subseteq GL(2,\mathbb F)$ is a finite group generated by
pseudoreflections then $S(V)^G$ is polynomial.
\end{theorem}

\begin{remark}
In general assertion (e) in Theorem \ref{structure} does not hold
for the case $n_r=2$. Let $\mathbb F=\mathbb F_4$ and $a$ an element
in $\mathbb F_4\setminus\mathbb F_2$. Define $\sigma$ by
$$\begin{array}{cc}
\sigma(1,1)=\mathbb F_2, & \sigma(1,2)=\{a\},\\
\sigma(2,1)=\{1\}, & \sigma(2,2)=\mathbb F_2.
\end{array}$$
Thus $t=1$ here. It is easy to calculate $|GL^{\sigma}(2,\mathbb
F_4)|=10$. This rules out $k_{11}=\mathbb F_2$ in assertion (e)
since $|GL(2,\mathbb F_2)|=6$, and it rules out $k_{11}=\mathbb F_4$
in assertion (e) since $|SL(2,\mathbb F_4)|=60$.
\end{remark}

\end{document}